\documentclass{article}
\usepackage{lineno}
\usepackage{bsinclude}

\title{\textbf{The McKay correspondence, tilting equivalences, and rationality}}
\author{
Morgan Brown
	\footnote{University of Miami, Coral Gables, Florida. {\sf m.brown@math.miami.edu}}, and
Ian Shipman
	\footnote{University of Utah, Salt Lake City, Utah. {\sf ian.shipman@gmail.com}}
}

\def\Hilb{\mathrm{Hilb}}
\def\wtt{\wt{\theta}}
\def\pp{{\prime \prime}}

\newcommand*\patchAmsMathEnvironmentForLineno[1]{%
  \expandafter\let\csname old#1\expandafter\endcsname\csname #1\endcsname
  \expandafter\let\csname oldend#1\expandafter\endcsname\csname end#1\endcsname
  \renewenvironment{#1}%
     {\linenomath\csname old#1\endcsname}%
     {\csname oldend#1\endcsname\endlinenomath}}%
\newcommand*\patchBothAmsMathEnvironmentsForLineno[1]{%
  \patchAmsMathEnvironmentForLineno{#1}%
  \patchAmsMathEnvironmentForLineno{#1*}}%
\AtBeginDocument{%
\patchBothAmsMathEnvironmentsForLineno{equation}%
\patchBothAmsMathEnvironmentsForLineno{align}%
\patchBothAmsMathEnvironmentsForLineno{flalign}%
\patchBothAmsMathEnvironmentsForLineno{alignat}%
\patchBothAmsMathEnvironmentsForLineno{gather}%
\patchBothAmsMathEnvironmentsForLineno{multline}%
}

\begin{document}
\maketitle

\begin{abstract}
We consider the problem of comparing t-structures under the derived McKay correspondence and for tilting equivalences.  We relate the t-structures using certain natural torsion theories.  As an application, we give a criterion for rationality for surfaces with a tilting bundle.  In particular we show that every smooth projective surface which admits a full, strong, exceptional collection of line bundles is rational.
\end{abstract}


\section{Introduction}
The McKay correspondence is a beautiful body of work comparing the representation theory of finte subgroups $G < \SL_n(\C)$ with the geometry of crepant resolutions of the singular quotient $\C^n/G$.  In this article we will study the approach to the McKay correspondence via derived categories (see \cite{KV,BKR}).  If $G < \SL_n(\C)$ where $n = 2,3$, then for each crepant resolution $Y \to \C^n/G$, there is a natural exact equivalence $\bD(\C^n)^G \cong \bD(Y)$ where $\bD(\C^n)^G$ is the bounded derived category of $G$-equivariant coherent sheaves on $\C^n$.  There has been much interest recently in understanding which objects of $\bD(Y)$ correspond to natural objects in $\bD(\C^n)^G$.  For example, in \cite{CCL,BCQV} the authors calculate the cohomology of objects corresponding to the equivariant sheaves $\cO_0 \tensor V$ on $\C^3$, where $V$ is an irreducible representation of $G$ and $G$ is abelian, and relate the results to certain combinatorics for describing the Hilbert scheme of $G$-orbits in $\C^3$ known as Reid's recipe.

We provide an alternative analysis of the equivalences appearing the derived McKay correspondence.  Instead of computing the derived McKay correspondence for particular objects, we instead describe the entire category $\coh(Y)$ as a full subcategory of $\bD(\C^n)^G$.  In general, given an equivalence $\Phi:\bD(X) \to \bD(\cat{A})$, where $X$ is a smooth projective variety and $\cat{A}$ is an Abelian category of interest, we can use $\Phi$ to identify $\coh(X)$ with a full subcategory $\Phi(\coh(X))$ of $\bD(\cat{A})$.  Then one can ask: what is the relationship between $\Phi(\coh(X))$ and $\cat{A}$?  

If the cohomology of every object of $\Phi(\coh(X))$ were concentrated in degrees $0$ and $1$ then the category $\Phi(\coh(X))$ is completely determined by a torsion pair in $\cat{A}$ and a construction known as tilting.  For the equivalences that show up in the McKay correpsondince in dimensions two and three, the cohomologies of objects corresponding to sheaves on $Y$ are in degrees 0, 1, and 2.  We extend the classical tilting picture by describing how to relate $\coh(X)$ and $\cat{A}$ by performing two tilts.  In order to state our results, we recall the definition of torsion pair.

\begin{definition*}
Let $\cat{A}$ be an Abelian category.  A pair $(\bT,\bF)$ of full subcategories is called a torsion pair if 
\begin{enumerate}
\item $\Hom(t,f) = 0$ for all $t \in \bT$ and $f \in \bF$ and
\item every $a \in \cat{A}$ fits into a short exact sequence $0 \to a' \to a \to a'' \to 0$ where $a' \in \bT$ and $a'' \in \bF$.
\end{enumerate}
\end{definition*}  

Suppose that $\Phi:\bD(\cat{A}) \to \bD(\cat{B})$ is an equivalence with the property that $\sfH^i(\Phi(a)) =0$ for all $a \in \cat{A}$ and $i \neq 0,1$.  Then we can define a torsion pair in $\cat{B}$ as follows \cite{BeRe}.  Put $\bF = \Phi(\cat{A}) \cap \cat{B}$ and $\bF = \Phi(\cat{A})[-1] \cap \cat{B}$.  Moreover, a torsion pair can be used to construct a new Abelian category \cite{HRS}, which in the case of the torsion pair defined above, recovers $\Phi(\cat{A})$ (as the upper tilt):

\begin{definition*}
Let $(\bT,\bF)$ be a torsion pair in an Abelian category $\cat{B}$.  The \emph{upper tilt} of $\cat{B}$ with respect to $(\bT,\bF)$ is the full subcategory $\cat{B}'$ of $\bD(\cat{B})$ consisting of objects $b^\bt$ such that $\sfH^0(b^\bt) \in \bF$, $\sfH^1(b^\bt) \in \bT$, and $\sfH^i(b^\bt) = 0$ for $i \neq 0, 1$.  We define the \emph{lower tilt} to be $\cat{B}'[1]$.
\end{definition*}

We now come to the definition that makes it possible to give uniform statements of our results.  The following definition is inspired by recent work on Bridgeland stability conditions (see \cite{BMT,BBMT,AB}, for example).  Let $\cat{A}$ be an Noetherian Abelian category.  A \emph{weak central charge} on $\cat{A}$ is a homomorphism $\sfZ:K_0(\cat{A}) \to \C$ such that for all $a \in \cat{A}$, (i) $\mathrm{Im}(\sfZ(a)) \geq 0$ and (ii) if $\mathrm{Im}(\sfZ(a)) = 0$ then $\mathrm{Re}(\sfZ(a)) \geq 0$.  Given a weak central charge $\sfZ=\theta + i \psi$ on $\cat{A}$, for any object $a \in \cat{A}$ there a maximal subobject $a_0 \subset a$ such that $\psi(a_0)=0$.  Now, we can form a torsion pair $(\bT_{\sfZ},\bF_{\sfZ})$ in $\cat{A}$ whose torsion part is 
\[ \bT_Z = \{ a \in \cat{A} : \text{for every quotient $a \onto a''$, either $a'' = a''_0$ or $\theta(a''/a''_0) > 0$} \}.\] 
We find that variants of this construction show up naturally in the study of derived equivalences in low dimension.

\begin{definition*}
The upper and lower tilts with respect to a weak central charge are defined to be the upper and lower tilts with respect to $(\bT_\sfZ,\bF_\sfZ)$.  
\end{definition*}

Using this definition we now consider the derived McKay correspondence in dimensions $2$ and $3$. Let $\theta:K_0(\C G) \to \Z$ be a homomorphism such that $\theta(\C G) =0$.  This can be used to define a notion of (semi)stability for finite length $G$-equivariant sheaves on $\C^n$.  A finite length $G$-equivariant sheaf $\cF$ on $\C^n$ is semistable if for any subsheaf $\cF' \subset \cF$, $\theta( \rmH^0(\cF') ) \geq \theta( \rmH^0(\cF) )$, where their global sections are viewed as representations of $G$.  A finite length $G$-equivariant sheaf is called a $G$-constellation if $\rmH^0(\cF) \cong \C G$ as $G$-representations.  The fine moduli space $\cM_\theta$ of $\theta$-stable $G$-constellations exists and if it is nonempty, it is a crepant resolution of $\C^n/G$.  Moreover, every crepant resolution arises this way (see \cite{BKR,CI}).  In addition, if $\cF \in \coh(\C^n \times \cM_\theta)$ is the universal $\theta$-stable $G$-constellation the Fourier-Mukai transform $\Phi_{\cF}:\bD(\cM_\theta) \to \bD(\C^n)^G$ is an equivalence.  We describe the relationship between $\coh(\cM_\theta)$ and $\coh(\C^n)^G$ using a torsion pair $(\bT_\theta,\bF_\theta)$ where
\[ \bT_\theta = \{ \cF : \text{ $\theta(\rmH^0(\cG))>0$ for all nonzero, finite length, $G$-equivariant quotients $\cF \onto \cG$} \}.\] 
One can think of $(\bT_\theta,\bF_\theta)$ as the maximal way to extend to $\coh(\C^n)^G$ the torsion pair in the category of finite length $G$-equivariant sheaves on $\C^n$ defined by the weak central charge $\sfZ = -\theta\circ\rmH^0 + i \ell$, where $\ell$ is the length function.  See Thm \ref{thm:2dMcKay} and Thm \ref{thm:main.McKay} for details. 

\begin{introthm}{A} \label{intro.McKay}
Let $\Phi:\bD(\cM_\theta) \to \bD(\C^n)^G$ be the Fourier-Mukai equivalence defined by the universal $\theta$-stable $G$-constellation.  There is a tilt $\cat{A}'$ of $\coh(\cM_\theta)$ that $\Phi$ identifies with the tilt of $\bD(\C^n)^G$ with respect to $(\bT_\theta,\bF_\theta)$.
\end{introthm} 

We also consider tilting equivalences.  In this situation, we obtain similar results but under an additional hypothesis.  In addition to describing the structure of a tilting equivalence, we give a new approach for establishing that varieties with a tilting bundle have Kodaira dimension $-\infty$.

A tilting bundle $\cE$ on a smooth projective variety $X$ is a vector bundle such that $\bR^i \Hom(\cE,\cE) = 0$ for $i \neq 0$ and (ii) $\cE$ generates the derived category of $X$.  Let $\sfA = \End(\cE)$ and regard $\cE$ as a sheaf of $\sfA$-modules.  In this situation there are mutually inverse equivalences
\[ 
\xymatrix{
\bD(X) \ar@<2pt>[r]^{\bR\Hom(\cE,?)} & \ar@<2pt>[l]^{? \tensor^{\bL} \cE} \bD(\sfA) 
}
\]
There is also a notion of stability for $\sfA$-modules.  Let $\theta:K_0(\sfA) \to \Z$ be a homomorphism.  Then a finite dimensional $\sfA$-module $M$ is $\theta$-semistable if for any submodule $M' \subset M$, $\theta(M') \geq \theta(M)$.  We say that $\theta$ is \emph{compatible} with $\cE$ if for any $p \in X$, $\Hom(\cE,\cO_p)$ is a $\theta$-stable $\sfA$-module.  See Theorem \ref{thm:main.tilting} for a more detailed version of the following.

\begin{introthm}{B} \label{intro.tilting}
Let $X$ be a smooth projective surface and $\Phi:\bD(\sfA) \to \bD(X)$ be a tilting equivalence.  Assume that there exists a $\theta$ compatible with $\cE$.  Then $X$ is rational and there are weak central charges on $\rmod\sfA$ and $\coh(X)$ whose tilts are identified by $\Phi$.
\end{introthm}

Unfortunately it is difficult to check the hypotheses of Thm B except in specific examples that are known already to be rational. But we were able by other means to establish rationality in the case where the surface has a tilting bundle which is a direct sum of line bundles, which is equivalent to having a full strong exceptional collection of line bundles. 

\begin{introthm}{C}\label{intro.linebundles}[Thm \ref{thm:linebundles}]
Any smooth projective surface which admits a full, strong, exceptional collection of line bundles is rational.

\end{introthm}
Finally, we would like to point out that one can obtain similar results for Fourier-Mukai equivalences between surfaces.  For example, in \cite{Hu2}, Huybrechts considers derived-equivalent K3 surfaces $X$ and $Y$.  He shows that after making a suitable choice of equivalence $\bD(X) \cong \bD(Y)$ there is a numerical tilt of $\coh(X)$ which is indentified with a numerical tilt of $\bD(Y)$.  

\subsection*{Acknowledgements}
We would like to thank Pieter Belmans, Alistair Craw, Bill Fulton, Lutz Hille, Alistair King, Robert Lazarsfeld, Emanuele Macr\`{i}, Markus Perling, and Nick Proudfoot for interesting conversations. We also thank Katrina Honigs and Jason Lo for comments and corrections. We are especially grateful to Emanuele Macr\`{i} for making us aware of the work of Huybrechts \cite{Hu2}. The second author would like to thank the Mathematical Sciences Research Institute for providing a wonderful atmosphere in which to work during part of this project.  During the completion of this work, the first author was partially supported by NSF RTG grant DMS-0943832, and the second author was partially supported by the NSF award DMS-1204733.  

\section{Preliminaries}
Let $\bk$ be an algebraically closed field of characteristic zero.  We make the standing assumption that the Abelian or triangulated categories under consideration are $\bk$-linear.  Furthermore, we adopt the convention that a full subcategory is closed under isomorphism of objects.

Suppose that $\cat{A}$ is an Abelian category.  Then we write $\bD(\cat{A})$ for the bounded derived category of $\cat{A}$.    While $\bD(\cat{A})$ is a triangulated category, it also has extra structure coming from the fact that it is a derived category; it carries a t-structure.  (See \cite{BBD} for details.)

\begin{definition}\label{def:tstructure}
Let $\bD$ be a triangulated category.  A \emph{t-structure} on $\bD$ is a pair $\alpha = (\cU, \cV)$ of full subcategories of $\bD$ such that
\begin{enumerate}
\item $\cU[1] \subset \cU$ and $\cV \subset \cV[1]$,
\item if $x \in \cU$ and $y \in \cV$ then $\Hom(x[1],y) = 0$,
\item for any $x \in \bD$ there exists a triangle
\[ \tau^{\leq 0} x \to x \to \tau^{> 0}x \to \tau^{\leq 0} x [1]\] 
where $\tau^{\leq 0} x \in \cU$ and $\tau^{> 0}x[1] \in \cV$.
\end{enumerate}
If $\alpha = (\cU,\cV)$ is a t-structure on $\bD$ then we define $\bD^{\leq n}_\alpha = \cU[-n]$ and $\bD^{\geq n}_\alpha = \cV[-n]$.
\end{definition}

Let $\bD$ be a triangulated category with a t-structure $\alpha = (\cU,\cV)$.  The subcategories $\cU$ and $\cV$ define one another.  Indeed, $\cV = (\cU[1])^\perp$, the full subcategory of objects $y \in \bD$ such that for any $x \in \cU$, $\Hom(x[1],y) = 0$.  Similarly, $\cU ={}^\perp(\cV[-1])$.  The full subcategory $\cA_\alpha = \cU \cap \cV$ is an Abelian category called the \emph{heart} of the t-structure.  Exact triangles in $\bD$ give rise to long exact sequences in the heart of a t-structure.  To see how, we observe that condition (3) in Definition \ref{def:tstructure} is equivalent to the existence of a left adjoint $\tau^{\leq 0}$ to the inclusion of $\cU \into \bD$, which is in turn equivalent to the existence of a right adjoint $\tau^{\geq 0}$ of $\cV \into \bD$.  So there is a functor $\sfH^0:\bD \to \cA_\alpha$ defined by $\sfH^0(x) = \tau^{\geq 0} \tau^{\leq 0}(x)$.  We define $\sfH^i(x) = \sfH^0(x[i]) \in \cA_\alpha$.  If 
\[ x' \to x \to x^\pp \to x[1]\] 
is an exact triangle in $\bD$ then applying $\sfH^0$ we get a long exact sequence
\[ \dotsm \to \sfH^0(x') \to \sfH^0(x) \to \sfH^0(x^\pp) \to \sfH^1(x') \to \dotsm \]
If we wish to emphasize the t-structure we will indicate it with a subscript, e.g. $\sfH^i_\alpha, \tau^{\leq 0}_\alpha, \tau^{>0}_\alpha,$ etc.

There is a natural t-struture on $\bD(\cat{A})$ where $\cU = \bD^{\leq 0}(\cat{A})$ and $\cV = \bD^{\geq 0}(\cat{A})$ are the full subcategories of objects quasi isomorphic to complexes supported in non-postive and non-negative cohomological degrees, respectively.  This is known as the \emph{standard} t-structure.  The standard t-structure on the bounded derived category satisfies an important finiteness condition.  If $\alpha$ is a t-structure on $\bD$ then we say it is \emph{bounded} if for any $x \in \bD$ there is an $n \geq 0$ such that $x \in \cU[-n] \cap \cV[-n]$.  The standard t-structure on $\bD(\cat{A})$ is bounded.  Furthermore, the natural inclusion $\cat{A} \to \bD(\cat{A})$ identifies $\cat{A}$ with the heart of the standard t-structure.  

\begin{rmk}
If $\alpha$ is a bounded t-structure on $\bD$, it is not neccessarily true that $\bD$ is equivalent to $\bD(\cA_\alpha)$.
\end{rmk}

Suppose that $\Phi:\bD(\cat{A}) \to \bD(\cat{B})$ is an exact equivalence.  Then the standard t-structure on $\bD(\cat{A})$ induces a t-structure on $\bD(\cat{B})$.  
\begin{definition} We denote the t-structure on $\bD(\cat{B})$ induced by $\Phi$ by $\alpha_\Phi = (\cU_\Phi,\cV_\Phi)$ where $x \in \cU_\Phi$ if and only if $x \cong \Phi(x')$ where $x' \in \bD^{\leq 0}(\cat{A})$ (and similarly for objects of $\cV_\Phi$).
\end{definition}
Since $\Phi$ equips $\bD(\cat{B})$ with two t-structures, one way to study $\Phi$ is to compare these t-structures.  In this general setting, the relationship could be very complicated.  However, for equivalences arising in algebraic geometry it is possible to describe their relationship precisely in dimensions one and two.  In order to describe this relationship, we need the notion of a torsion pair.  We recall the definition from the introduction.
\begin{definition} \label{def:torsion_pair}
Let $\cat{A}$ be an Abelian category.  A \emph{torsion pair} in $\cat{A}$ is a pair $(\bT,\bF)$ of full subcategories such that
\begin{enumerate}
\item for each $t \in \bT$ and $f \in \bF$, $\Hom(t,f) = 0$, and
\item each $a \in \cat{A}$ fits into an exact sequence $0 \to a^\prime \to a \to a^\pp \to 0$ where $a^\prime \in \bT$ and $a^\pp \in \bF$.
\end{enumerate}
\end{definition}

If $\bD$ is a triangulated category with a t-structure, then a torsion pair in its heart determines a new t-structure.  More precisely, we have the following theorem.
\begin{thm}[\cite{HRS}] \label{thm:HRS}
Let $\bD$ be a triangulated category with a t-structure $\alpha = (\cU,\cV)$ and suppose that $\pi=(\bT,\bF)$ is a torsion pair in its heart $\cA_\alpha$.  Then the pair $\alpha_\pi = (\cU_\pi,\cV_\pi)$ defined by
\[ \cU_\pi = \{ x \in \cU[-1] : \sfH^1_\alpha(x) \in \bT \}\] 
and
\[ \cV_\pi = \{ x \in \cV : \sfH^{0}_\alpha(x) \in \bF \}\] 
is a t-structure on $\bD$.
\end{thm}

There is a certain ambiguity in the definition of the t-structure $\alpha_\pi$, for we could just as easily have defined it so that the objects in the heart have nonzero cohomology only in cohomological degrees 0 and 1 rather than -1 and 0.  So we will simply absorb this ambiguity in our terminology.  Let $\alpha, \beta$  be two t-structures.  We say that $\alpha$ is a \emph{tilt} of $\beta$ if there exists an $n \in \Z$ such that for any $\sfH^i_{\alpha}(x) =0$ for any $x \in \cA_{\beta}$ and $i \neq n,n+1$.  If $\alpha$ is a t-structure, let us denote by $\alpha[n]$ the t-structure $(\cU[n], \cV[n])$.  Woolf points out the following connection between torsion pairs and t-structures related by tilts \cite[Prop. 2.1]{Wo}.

\begin{lem}\label{lem:onetilt}
A t-structure $\alpha$ is a tilt of $\beta$ if and only if there is a torsion pair $\pi$ in $\cA_{\alpha}$ such that $\alpha = \beta_\pi$ up to a shift.
\end{lem}

We will show that tilting and McKay equivalences in dimensions two and three induce pairs of t-structures on $\bD(X)$ that are related by at most two tilts.  For a tilting equivalence, it is possible to show this on abstract grounds (see \cite{Lo}).  However, it turns out that one can give an explicit description of the tilts.  They have appeared as the first of two tilts in several works on Bridgeland stability conditions (see \cite{AB,BMT,BBMT}, following \cite{Br08}).  These torsion pairs are defined using Mumford's slope stability.  There is a similar notion, King stability \cite{Ki}, for objects in finite length categories.  We will work in a framework, a weakened form of Bridgeland's stability conditions, that allows us to treat both of these uniformly.

Let $\bD$ be a triangulated category with a bounded t-structure $\alpha$.  Then the natural inclusion $\cA_\alpha \to \bD$ induces an isomorphism of Grothendieck groups $K_0(\cA_\alpha) \to K_0(\bD)$.  An \emph{additive} function on an Abelian category is a homomorphism from its Grothendieck group to an Abelian group, which will be $\C$ in this paper.  Now, we recall the definition of weak central charge from the introduction:

\begin{definition}
Let $\cat{A}$ be a Noetherian abelian category.  An additive $\C$-valued function $\sfZ$ on $\cat{A}$ is a weak central charge if
\begin{enumerate}
\item for all $a \in \cat{A}$, $\mathrm{Im} (\sfZ(a)) \geq 0$, and
\item if $\mathrm{Im}(\sfZ(a)) = 0$ then $\mathrm{Re}(\sfZ(a)) \geq 0$.
\end{enumerate} 
\end{definition}

Let $\bD$ be a triangulated category with a t-structure $\alpha$ whose heart is Noetherian and suppose that $\sfZ = \theta + i \psi$ is a weak central charge on $\cA_\alpha$.  Polishchuk pointed out that if a full subcategory of a Noetherian category is closed under quotients and extensions then it is the torsion part of a torsion pair \cite{Po}.  We note that since $\psi$ is non-negative, the full subcategory $\cat{A}_0$  of $\cat{A}$ consisting of objects $a$ with $\psi(a) = 0$ is a Serre subcategory.  In particular, there is a torsion pair $(\bT_\psi,\bF_\psi)$ such that 
\[ \bT_\psi = \{ a \in \cat{A} : \psi(a) = 0 \}\] 
We use this to construct a torsion pair $\pi_\sfZ = (\bT_\sfZ,\bF_\sfZ)$.  The torsion part is defined by
\[\bT_\sfZ = \{ x \in \cat{A} : \text{ for every quotient $x \onto x^\pp$, either $x^\pp_{\bF_\psi} = 0$ or $\theta( x^\pp_{\bF_\psi} ) > 0$  } \}.\] 
By Theorem \ref{thm:HRS}, $\pi_\sfZ$ determines a t-structure on $\bD$, which we denote by $\alpha_\sfZ$. 
\begin{definition}
We say that a t-structure $\beta$ is a Harder-Narasimhan tilt, or HN tilt, of $\alpha$ if there exists a weak central charge $\sfZ$ on $\cA_\alpha$ such that $\beta = \alpha_{\sfZ}$ up to a shift.
\end{definition}

\begin{rmk}
Suppose that $\alpha$ is a t-structure on $\bD$ and that $\pi$ is a torsion pair in $\cA_\alpha$.  Even if $\cA_\alpha$ is Noetherian there is no reason for $\cA_{\alpha_\pi}$ to be Noetherian.  However, one can show that if $\pi$ is the torsion pair derived from a weak central charge then $\cA_{\alpha_\pi}$ is in fact Noetherian.
\end{rmk}

The first examples of weak central charges come from notions of stability in abelian categories. We first consider King stability: Suppose that $\sfA$ is a finite dimensional associative $\bk$-algebra.  We write $\rmod\sfA$ for the category of finite dimensional right $\sfA$-modules.  Then $K_0(\rmod\sfA)$ is freely generated by the finitely many isomorphism classes of simple $\sfA$-modules.  

Now let $\theta$ be any additive function $K_0(\rmod\sfA)\to \mathbb{R}$. If $M \in \rmod\sfA$ satisfies $\theta(M)=0$, we say that $M$ is $\theta$-semistable in the sense of King \cite{Ki} if for any submodule $M' \subset M$, $\theta(M')\leq 0$, and $M$ is $\theta$-stable if $\theta(M')<0$ whenever $M'$ is a nonzero proper submodule of $M$.

Note that we have adopted the opposite sign convention from King. This is so we can treat categories of modules and sheaves in a uniform manner.

Let $\ell: K_0(\rmod\sfA)\to \mathbb{R}$ be the length function, and for $s\in \mathbb{R}$, set $\theta_s(M)=\theta(M)-s \ell(M)$. Then we say that an arbitrary finitely generated module $M$ is $\theta$-(semi)stable if and only if it is $\theta_s$-(semi)stable for $s = \theta(M)$, as then $\theta_s(M)=0$. Now, for any $s$, we can define a weak central charge $\sfZ_s = \theta_s + i \ell$. This allows us to assign a slope $\mu_s$ to a nonzero object $M$ by the formula $\mu_s(M)=\frac{\theta_s(M)}{\ell(M)}=\frac{\theta_s(M)}{\ell(M)}-s$, and we also write $\mu(M)$ for $\mu_0(M)$. Using the composition series for $M$, we have a unique Harder-Narisimhan filtration

\[ 0=M_0 \subset \dotsm \subset M_j = M\] 
where each $M_i/M_{i-1}$ is $\theta$-semistable, and
\[ \mu_{max}(M) = \mu_0(M_1/M_0) > \dotsm > \mu_0(M_j/M_{j-1}) = \mu_{min}(M).\] 

The two subcategories comprising the torsion pair $(\bT_\sfZ,\bF_\sfZ)$ can be defined purely in terms of this filtration. A nonzero object $M$ is in the torsion part $\bT_\sfZ$ if $\mu_{min}(M)>0$, and $M$ is in the free part $\bF_\sfZ$ if $\mu_{max}(M) \leq 0$.

This notion of stability is closely related to GIT; the function $\theta_s$ corresponds to a character of a linear algebraic group acting on a space of quiver representations, and a module is $\theta$-(semi)stable if and only the associated representation is (semi)stable in the corresponding GIT problem \cite{Ki}.

A similar phenomenon occurs for the category of coherent sheaves on an algebraic variety: Consider a smooth projective variety $X$ of dimension $n$ with an ample line bundle $\cL$. Then for any $s \in \R$ we can produce an additive function $\theta_s: K_0(X) \to \mathbb{R}$ such that for a coherent sheaf $\cE$, $\theta_s(\cE) = c_1(\cE) \cdot c_1(\cL)^{n-1} - s \cdot \rk(\cE)$. Then $\sfZ_s = \theta_s + i \rk$ is a weak central charge on $\coh(X)$ because if $\rk(\cE)=0$, then $\cE$ is torsion and the intersection number $c_1(\cE) \cdot c_1(\cL)^{n-1}=0$.

As in the case of King stability, the central charge allows us to define a slope by 
\[\mu_s(\cE)=\frac{\text{Re}(Z_s(\cE))}{\text{Im}(Z_s(\cE)}=\frac{c_1(\cE) \cdot c_1(\cL)^{n-1}}{\rk{\cE}}.\]
This is well defined for torsion-free sheaves, and when $s=0$ we recover the usual definition of $\mu$, the slope of a torsion-free sheaf with respect to an ample divisor. A torsion free sheaf $\cE$ is slope semi-stable if $\mu(\cE') \leq \mu(\cE)$ for any subsheaf $\cE' \subset \cE$.  Now, it is well known that a torsion-free sheaf admits a unique Harder-Narasimhan filtration
\[0=\cE_0 \subset \dotsm \subset \cE_r = \cE\] 
where $\cE_{i+1}/\cE_i$ are torsion free semistable sheaves and 
\[ \mu_{max}(\cE) = \mu(\cE_1/\cE_0) > \dotsm > \mu(\cE_r/\cE_{r-1}) = \mu_{min}(\cE).\] 
In fact $\cE$ is semistable if and only if the filtration is given by $0 \subset \cE$, which holds if and only if $\mu_{max}(\cE)=\mu_{min}(\cE)$.
Then $(\bT_{\sfZ_s},\bF_{\sfZ_s})$ can be described as follows.  A sheaf $\cE$ belongs to $\bT_{\sfZ_s}$ if and only if it is torsion or $\mu_{min}(\cE^\pp) >s$, where $\cE^\pp$ is the maximal torsion-free quotient of $\cE$.   Similarly, a sheaf $\cE$ belongs to $\bF_{\sfZ_s}$ if and only if it is torsion free and $\mu_{max}(\cE) \leq s$.

We now come to the technical heart of the paper.  We consider an equivalence $\Phi:\bD(X) \to \bD(\cat{A})$.  It is convenient to use the notation $\Phi^i(\cF^\bt) = \sfH^i(\Phi(\cF^\bt))$ for $\cF^\bt \in \bD(X)$.  Recall that $\Phi$ is \emph{left exact} if $\Phi^i(\cF) = 0$ for all $i < 0$ and $\cF \in \coh(X)$.  We will restrict our attention to functors that satisfy the standard vanishing theorems.  For a sheaf $\cF \in \coh(X)$ let $\dim(\cF)$ denote the dimension of the support of $\cF$.
\begin{definition}\mbox{}
\begin{enumerate}
\item We say that $\Phi$ \emph{satisfies Grothendieck vanising (GV)} if $\Phi^i(\cF) = 0$ for $i > \dim(\cF)$.
\item We say that $\Phi$ \emph{satisfies Serre vanishing (SV)} if for any ample line bundle $\cL$ and sheaf $\cF$ there is an $n_0$ such that if $n > n_0$, $\Phi^i(\cF \tensor \cL^{\tensor n}) = 0$ for $i > 0$.
\end{enumerate}
\end{definition}

Note that if $\Phi$ is left exact and satisfies (GV) then $\Phi(\cO_p) \in \cat{A}$ for all $p \in X$.  The following criterion will allow us to understand the relationship between $\Phi(\coh(X))$ and $\cat{A}$ if we understand the objects $\Phi(\cO_p)$.  Let $W \subset X$ be a closed subset and denote by $\bD_W(X)$ the derived category of coherent sheaves supported on $W$.  Note that $\bD_W(X)$ is naturally a full subcategory of $\bD(X)$ and $\bD_X(X) = \bD(X)$.
\begin{lem} \label{lem:main}
Let $X$ be a quasiprojective variety, $W \subset X$ a closed set, and $\Phi:\bD_W(X) \to \bD(\cat{A})$ be a left exact equivalence that satisfies (GV) and (SV).  Suppose that $(\bT,\bF)$ is a torsion pair in $\cat{A}$ such that
\begin{enumerate}
\item for every point $p \in W$, $\Phi(\cO_p) \in \bF$, and 
\item if $\cF$ is a zero-dimensional sheaf supported on $W$ and $\Phi^0(\cF) \onto x$ is a surjection with $x \in \bF$ nonzero then there is a $p \in X$ and a nonzero map $x \onto \Phi^0(\cO_p)$.
\end{enumerate}
Then for a sheaf $\cE$ supported on $W$ we have $\Phi^0(\cE) \in \bF$ and $\Phi^m(\cE) \in \bT$ where $m = \dim(\cE)$.
\end{lem}
\begin{proof}
We fix an ample line bundle $\cL$ on $X$.  By assumption, if $p \in X$, then $\Phi(\cO_p) = \Phi^0(\cO_p) \in \bF$. Suppose that $\cE$ is supported on a subscheme of dimension $m > 0$.  Then there is an $N > 0$ and a section $s:\cO \to \cL^{\tensor N}$ such that $\cE \to \cE \tensor \cL^{\tensor N}$ is injective and its cokernel $\cF_s$ satisfies $\dim(\cF_s) = m-1$.  Moreover, we can choose $N$ large enough that $\Phi^i(\cE \tensor \cL^{\tensor N}) = 0$ for $i > 0$ and thus $\Phi^{m-1}(\cF_s) \to \Phi^m(\cE)$ is surjective.  Since $\bT$ is closed under quotients, to show $\Phi^m(\cE) \in \bT$ it suffices to show this when $\dim(\cE)=1$.   

So suppose that $\dim(\cE)=1$.  Let $\cE^\pp$ be the maximal quotient of $\Phi^1(\cE)$ lying in $\bF$.  Assume for a contradition that $\cE^\pp \neq 0$.  Once again for some $N \gg 0$ a general section $s:\cO_X \to \cL^{\tensor N}$ gives rise to an exact sequence 
\[ 0 \to \cE \to \cE \tensor \cL^{\tensor N} \to \cF_s \to 0.\] 
such that the boundary map $\Phi^0(\cF_s) \to \Phi^1(\cE)$ is surjective.  So $\cE^\pp$ is a quotient of $\Phi^0(\cF_s)$ and, by hypothesis, there is a $p_0 \in X$ and a nonzero morphism $\cE^\pp \to \Phi^0(\cO_{p_0})$. However, we can vary the section $s$ so that $p_0$ does not belong to the support of $\cF_s$.  The map $\Phi^0(\cF_s) \to \Phi^0(\cO_{p_0})$ is nonzero, yet it is induced by a map $\cF_s \to \cO_{p_0}$ that must be zero since $p_0$ is not in the support of $\cF_s$.  We conclude that $\Phi^1(\cF) \in \bT$.

We now turn to the assertion that $\Phi^0(\cE) \in \bF$ and proceed by induction on $\dim(\cE)$.  Consider the exact sequence 
\[ 0 \to \Phi^0(\cE \tensor \cL^{\tensor -N}) \to \Phi^0(\cE) \to \Phi^0( \cF_s \tensor \cL^{\tensor -N} ).\] 
Let $\cE^\prime \subset \Phi^0(\cE)$ be the maximal subobject in $\bT$.  Since $\Phi^0(\cF_s \tensor \cL^{\tensor -N}) \in \bF$ by induction we see that $\cE^\prime \to \Phi^0(\cE)$ must factor through $\Phi^0(\cE \tensor \cL^{\tensor -N})$.  Since $\Phi(\cE)$ is concentrated in non-negative cohomological degrees, there is a map $\Phi^0(\cE) \to \Phi(\cE)$ in $\bD(\cat{A})$ which induces a map $\cE^\prime \to \Phi(\cE)$.  Let $\Psi$ be the inverse equivalence to $\Phi$.  Then with $\cB^\prime = \Psi(\cE^\prime)$ we obtain a map $\cB^\prime \to \cE$ that factors, for large $N$, through the maps $\cE \tensor \cL^{\tensor -N} \to \cE$ induced by general sections.  Since $\cB^\prime$ is concentrated in non-positive homological degrees, the map $\cB^\prime  \to \cE$ factors through $\sfH^0(\cB^\prime)$.  Since $\cL$ is ample, it follows from the Krull intersection theorem the intersection of the subsheaves $\cE \tensor \cL^{\tensor -N} \to \cE$ is contained in the maximal finite length subsheaf $\cG \subset \cE$.  So $\cB^\prime \to \cE$ factors through $\cB^\prime \to \cG$.  Hence $\cE^\prime \to \Phi^0(\cE)$ factors through a map $\cE^\prime \to \Phi^0(\cG)$. Since $\Phi^0(\cG) \in \bF$ this map is zero and therefore $\cE^\prime = 0$.  Thus $\Phi^0(\cE) \in \bF$.
\end{proof}

The following Lemma is useful for analyzing equivalences $\Phi:\bD_Z(X) \to \bD(\cat{A})$ where $\cat{A}$ admits a natural weak central charge. 
\begin{lem}\label{lem:mainHN}
Let $X$ be a quasi-projective variety, $W$ a closed subset, and $\Phi:\bD_W(X) \to \bD(\cat{A})$ be a left exact equivalence that satisfies (GV) and (SV).  Suppose that $\sfZ$ is a weak central charge on $\cat{A}$ such that for every point $p \in W$, 
\begin{enumerate}
\item $\Phi^0(\cO_p) \in \bF_\psi$, 
\item $\theta(\Phi^0(\cO_p))=0$, but $\theta(x) > 0$ for every proper nonzero quotient $\Phi^0(\cO_p) \onto x$ where $x \in \bF_\psi$, and
\item if $\sfZ(x) = 0$, $\Ext^1(x,\Phi^0(\cO_p))=0$.
\end{enumerate}
Then for a sheaf $\cE$ supported on $W$ we have $\Phi^0(\cE) \in \bF_\sfZ$ and $\Phi^m(\cE) \in \bT_\sfZ$ where $m = \dim(\cE)$.
\end{lem}
\begin{proof}
We will deduce this from Lemma \ref{lem:main}.  More precisely, we must check that if $\cF \in \coh_W(X)$ is a sheaf of finite length and $\Phi^0(\cF) \onto x$ is a surjection with $x \in \bF_\sfZ$ then there is a point $p \in X$ and a surjection $x \onto \Phi^0(\cO_p)$.  In fact, we will show that in this situation, $x$ is an iterated extension of objects of the form $\Phi^0(\cO_p)$, $p \in X$. 

Suppose that $x \in \bF_\sfZ$.  We claim that $\theta(x) \leq 0$.  Assume not for a contradiction.  Since $\bT_\psi \subset \bT_\sfZ$ we see that $\psi$ is positive on all subobjects of $\bF_\sfZ$.  Let $x' \subset x$ be a subobject that minimizes $\psi$ among all subobjects with $\theta(x') > 0$.  Suppose that $x' \onto y$ is a quotient with $y \in \bF_\psi$.  Then $\psi(y) > 0$ so if $x^\pp \subset x'$ is the kernel of $x'\onto y$ then $\psi(x^\pp) < \psi(x')$.  Therefore, $\theta(x^\pp) \leq 0$ by minimality of $x'$.  So $\theta(y) \geq \theta(x') > 0$.  Hence $x' \in \bT_\sfZ$, which is absurd since $\Hom(\bT_\sfZ,\bF_\sfZ) = 0$.

We now proceed by induction on the length of $\cF$.  Choose a surjection $\cF \onto \cO_p$ and let $\cF'$ be its kernel.  Let $x'$ be the image in $x$ of $\Phi^0(\cF')$ under $\Phi^0(\cF) \to x$.  Then $x/x'$ is a quotient of $\Phi^0(\cO_p)$ and therefore $\theta(x/x') \geq 0$.  By induction, $x'$ is an iterated extension of objects of the form $\Phi^0(\cO_q)$.  So we see that $\theta(x') = 0$ and therefore $\theta(x/x') = 0$.  However, this implies that either $x/x' \cong \Phi^0(\cO_p)$ or $\psi(x/x') = 0$.  The possibility that $x/x'$ is a nonzero object and $\psi(x/x')=0$ is ruled out since (3) implies that the exact sequence
\[ 0 \to x' \to x \to x/x' \to 0\] 
would have to split, yet $x \in \bF_\psi$.
\end{proof}

We can use Lemma \ref{lem:onetilt} to detect whether a given t-structure on $\bD_W(X)$ can be reconstructed from a torsion pair on $\coh_W(X)$.  However, we would like to characterize when the t-structure is a HN tilt of the standard t-structure.  Let $\eta_1,\dotsc,\eta_p$ be the generic points of the components of $W$.  Then for any sheaf $\cF$ supported on $W$, we define 
\[ \rk_W(\cF) = \sum_{i=1}^p{ \mathrm{length}_{\cO_{X,\eta_i}}( \cF_{\eta_i} ) }.\] 
We call a sheaf $\cF$ with $\rk_W(\cF)=0$ \emph{torsion} and a sheaf is \emph{torsion-free} if it has no torsion subsheaves.
\begin{lem}\label{lem:reverseHN}
Let $\alpha$ be a t-structure on $\bD_W(X)$.  Suppose that for all $x \in \cA_\alpha$, $\sfH^i(x) = 0$ for $i \neq -1,0$.  The t-structure $\alpha$ is a HN tilt of the standard t-structure on $\bD_W(X)$ for the weak central charge $\sfZ = \theta + i \rk_Z$ if and only if 
\begin{enumerate}
\item every torsion sheaf belongs to $\cA_\alpha$, 
\item $\theta$ is non-negative on $\cA_\alpha$, and
\item if $\cF \in \coh_W(X) \cap \cA_\alpha$ is a sheaf such that $\theta(\cF) = 0$ then $\cF$ is torsion.
\end{enumerate}
\end{lem}
\begin{proof}
We begin with the observation that $\coh_W(X) \cap \cA_\alpha$ is a torsion class and hence closed under quotients.  Now, $\theta$ is non-negative on $\coh_W(X) \cap \cA_\alpha$.  So if $\cF \in \coh_W(X) \cap \cA_\alpha$ and $\cF^\pp$ is its free part, $\theta(\cF^\pp) > 0$ by condition (3).  So $\theta$ is positive on every free quotient.  Therefore $\coh_W(X) \cap \cA_\alpha \subset \bT_\sfZ$.  On the other hand, suppose that $\cF \in \bT_\sfZ$ and let $\cF' \subset \cF$ be the maximal subobject of $\cF$ belonging to $\coh_W(X) \cap \cA_\alpha$.  Since all torsion sheaves belong to $\cA_\alpha$, we see that $\cF^\pp = \cF/\cF'$ is a torsion-free sheaf.  Therefore if $\cF^\pp \neq 0$ then $\theta(\cF^\pp) > 0$.  However, $\cF^\pp[1] \in \cA_\alpha$ and since $\theta$ is non-negative on $\cA_\alpha$, $\cF^\pp = 0$.  We conclude that $\cF \in \coh_W(X) \cap \cA_\alpha$ and thus $\bT_\sfZ = \coh_W(X) \cap \cA_\alpha$.
\end{proof}

Suppose that $X$ is quasiprojective variety and $W \subset X$ is a surface.  Say that $\Phi:\bD_W(X) \to \bD(\cat{A})$ is a left exact equivalence where $\Phi^i(\cE) = 0$ for $\cE \in \coh_W(X)$ and $i > 2$.  Let $\pi = (\bT,\bF)$ be a torsion pair as in Lemma \ref{lem:main}.  Then $\alpha_\Phi$ is a tilt of $\beta_\pi$, where $\beta$ is the standard t-structure on $\bD(\cat{A})$, as in the discussion above.  It also follows that if $\cF$ is a torsion sheaf then $\Phi(\cF) \in \cA_{\beta_\pi}$.  So if there is an additive function $\theta$ on $\cat{A}$ that is non-negative on $\bT$ and non-positive on $\bF$, then $\theta$ induces a non-negative additive function on $\cA_{\beta_\pi}$.  Via $\Phi$ we can view $\theta$ as an additive function on $\coh(X)$.  Now $\Phi(\cF) \in \cA_{\beta_\pi}$ if and only if $\Phi^2(\cF) = 0$ and $\Phi^1(\cF) \in \bT$.  So by the previous Lemma, $\beta_\pi$ is an ordinary HN tilt of $\alpha_\Phi$ if and only if any sheaf such that $\theta(\cF) =0$, $\Phi^2(\cF) = 0$, and $\Phi^1(\cF) \in \bT$ is supported on a curve.

\section{The derived McKay correspondence}
Let $G \subset \SL_n(\C)$ be a nontrivial finite subgroup.  Then $G$ naturally acts on $\C^n$.  The categorical quotient $\C^n/G$ is singular and there is a tight connection between the representation theory of $G$ and the crepant resolutions of $\C^n/G$.  One approach to resolving the singularities on $\C^n/G$ is to consider spaces of stable $G$-equivariant sheaves on $\C^n$ of finite length.

\begin{definition}\mbox{}
\begin{enumerate}
\item A \emph{$G$-constellation} on $\C^n$ is a $G$-equivariant, finite length sheaf $\cF$ such that $\rmH^0(\cF) \cong \C G$ as representations of $G$.
\item A \emph{$G$-cluster} is a finite $G$-equivariant subscheme $Z \subset \C^n$ such that $\rmH^0(\cO_Z) \cong \C G$ as representations of $G$.
\end{enumerate}
\end{definition}

Note that a $G$-cluster is a $G$-constellation expressed as a $G$-equivariant quotient of $\C^n$.  There is a natural notion of stability for $G$-constellations, generalizing King stability.  Let $\theta:K_0(\C G) \to \Z$ be a homomorphism such that $\theta(\C G) = 0$.  Then we say that a $G$-constellation $\cF$ is \emph{$\theta$-semi-stable} if for every proper, nonzero quotient $\cF \onto \cF''$, $\theta(\cF'') \geq 0$.  For each $\theta$, there is a fine moduli space $\cM_\theta$ of $\theta$-stable $G$-constellations \cite{CI}.  There is a special additive function of $G$ representations, $\theta_0:K_0(\C G) \to \Z$, by 
\[
\theta_0(V) = \begin{cases} -\dim(V)^2 & \text{if $V$ is nontrivial,} \\ \# G - 1 & \text{if $V$ is trivial.} \end{cases}
\]
Then $\cM_{\theta_0} \cong G\Hilb(\C^n)$, the fine moduli space of $G$-clusters (see \cite{INm,INj,CI}).

We denote the category of $G$-equivariant coherent sheaves on $\C^n$ by $\coh(\C^n)^G$ and its derived category by $\bD(\C^n)^G$.  Of course, $\coh(\C^n)^G$ may be interpreted as the category of coherent sheaves on the stack quotient $[\C^n/G]$.  We can view the stack quotient $[\C^n/G]$ as a tautological crepant resolution of the categorical quotient $\C^n/G$.  One aspect of the (mostly conjectural) derived McKay correspondence is that $\bD(\C^n)^G$ should be equivalent to the derived category of any geometric crepant resolution of $\C^n/G$.  Let $\cE_\theta \in \coh(\C^n \times \cM_\theta)^G$ denote the universal $G$-constellation, where $G$ acts on $\C^n \times \cM_\theta$ via the first factor.  Then we consider the functor
\[ \Phi = \bR p_{\C^n *}( \cE_\theta \tensor p_{\cM_\theta}^*(?) ):\bD(\cM_\theta) \to \bD(\C^n)^G\] 
where $\bD(\C^n)^G$ denotes the derived category of $G$-equivariant coherent sheaves on $\C^n$.  Kapranov and Vasserot \cite{KV} proved that $\Phi$ is an equivalence in dimension $n=2$, and Bridgeland, King, and Reid \cite{BKR} proved that it is an equivalence in dimension $n=3$.

Suppose that $\theta:K_0(\C G) \to \Z$ is a homomorphism such that $\theta(\C G) = 0$.  Let $(\bT_\theta,\bF_\theta)$ be the  torsion pair with torsion class
\[ \bT_\theta = \{ \cF : \text{ for all nonzero finite length, $G$-equivariant quotients $\cF \onto \cG$, $\theta(\rmH^0(\cG)) > 0$ } \}. \] 
Consider a sheaf $\cF$ that is supported at a point $p$ such that $G_p$ is trivial.  Let $G \cdot p$ be the orbit of $p$ and $\cO_{G\cdot p}$ its structure sheaf. The category of $G$-equivariant coherent sheaves supported on $G \cdot p$ has a unique simple object, $\cO_{G \cdot p}$.  Therefore $\rmH^0( \cF \tensor \cO_{G \cdot p} ) = \C G ^{\oplus N}$ for some $N$.  It follows that $\wtt(\cF \tensor \cO_{G \cdot p}) = 0$.  Hence if $\cF \in \bT_\theta$ then $\cF$ is supported on the locus in $\C^n$ where the action of $G$ is not free.  

Suppose that $n=2$.  Then $\C^2/G$ is a Kleinian singularity and $\cM_\theta$ is either empty or equal to $G\Hilb(\C^2)$.  We note that in this situation $\bT_\theta$ consists entirely of $G$-equivariant sheaves supported at the origin.  Using the length function $\ell$ on the category of $G$-equivariant sheaves supported at the origin, one can construct HN filtrations.  Then $\bT_\theta$ consists of those $G$-equivariant sheaves supported at the origin whose HN factors all have positive slope.

\begin{thm}\label{thm:2dMcKay}
Under the equivalence $\Phi:\bD(G\Hilb(\C^2))\to \bD(\C^2)^G$ the standard t-structure on $\bD(G\Hilb(\C^2))$ is identified with the tilt of $\coh(\C^2)^G$ with respect to $(\bT_{\theta_0},\bF_{\theta_0})$.
\end{thm}
\begin{proof}
We note that $\Phi$ is left exact and satisfies SV and GV.  Let $E \subset G\Hilb(\C^2)$ be the exceptional divisor of the map $G\Hilb(\C^2)\to \C^2/G$.  Then $\Phi$ restricts to an equivalence $\Phi_0:\bD_E(G\Hilb(\C^2)) \to \bD_0(\C^2)^G = \bD( \coh_0(\C^2)^G )$.  Now, we notice that $(\bT_\theta,\bF_\theta)$ induces a torsion pair on the category of $G$-equivariant sheaves supported at the origin that agrees with $(\bT_\sfZ,\bF_\sfZ)$ where $Z = \theta_0 \circ \rmH^0 + i \ell$.  So we can apply Lemma \ref{lem:mainHN}.  Since $G$-clusters are $\theta_0$-stable, Lemma \ref{lem:mainHN} implies that if $\cF$ is a sheaf supported on $E$ then $\Phi^0(\cF) \in \bF_{\theta_0}$ and $\Phi^1(\cF) \in \bT_{\theta_0}$.  

For any $\cF \in \coh(G\Hilb(\C^2)$, $\Phi^i(\cF) = 0$ unless $i=0,1$.  Thus Lemma \ref{lem:onetilt} implies that there is a torsion pair $\pi= (\bT',\bF')$ in $\coh(\C^2)^G$ such that $\alpha_\Phi = \beta_\pi$, where $\beta$ is the standard t-structure on $\bD(\C^2)^G$.  Then $\bT' \subset \coh_0(\C^2)^G$.  We compute
\begin{align*}
\bT' & = \coh(\C^2)^G \cap \coh(G\Hilb(\C^2))[1]  \\ 
& = \coh_0(\C^2)^G \cap \coh(G\Hilb(\C^2))[1] \\
& = \coh_0(\C^2)^G \cap \coh_E(G\Hilb(\C^2))[1] \\
& = \bT_{\theta_0}, 
\end{align*}
where the third equality follows from the fact that $\Phi$ identifies $\bD_E(G\Hilb(\C^2))$ with $\bD_0(\C^2)^G$.
\end{proof}

We now turn to dimension three.  If $\C^3/G$ has an isolated singularity then $\bT_\theta$ once again consists of $G$-equivariant sheaves supported at the orgin and the torsion class $\bT_\theta$ can be described in terms of HN filtrations.  However if the singularity is not isolated, the torsion pair can be much more complicated.  Let $E \subset \cM_\theta$ be the part of the exceptional locus of $\cM_\theta \to \C^3/G$ lying over $0 \in \C^3/G$.  Let $\Phi:\bD(\cM_\theta) \to \bD(\C^3)^G$ be the equivalence defined above (whose Fourier-Mukai kernel is the universal $G$-constellation).  Let $\alpha$ and $\beta$ denote the standard t-structures on $\bD_E(\cM_\theta)$ and $\bD_0(\C^3)^G$, respectively.
\begin{thm}\label{thm:main.McKay}
The t-structures on $\bD(\C^3)^G$ induced by $\Phi$ and $\pi$ where $\pi = (\bT_\theta,\bF_\theta)$, respectively, are related by a (possibly trivial) tilt.  Moreover, $\Phi$ restricts to an equivalence between $\bD_E(\cM_\theta)$ and $\bD_0(\C^3)^G$.  For $\sfZ_\cM = -\theta + i \rk_E$ and $\sfZ_G = \theta + i \ell$ on $\coh_E(\C^3)$ and $\coh_0(\C^3)^G$, $\Phi_0$ identifies $\alpha_{\sfZ_\cM}[1]$ with $\beta_{\sfZ_G}$.
\end{thm}
\begin{proof}
We will prove the second statement first.  It is clear that $\cF$ is supported on $E$ if and only if $\Phi^i(\cF)$ is supported on the origin for all $i$.  Hence $\Phi$ restricts to an equivalence 
\[ \Phi_0 : \bD_E(\cM_\theta) \to \bD_0(\C^n)^G.\]
Now, let us apply Lemma \ref{lem:mainHN}.  Since $\Phi$ is left exact and satisfies (GV) and (SV), the same is true for $\Phi_0$.  Observe that $\ell(\cF) = 0$ if and only if $\cF = 0$.  So the first and third conditions of Lemma \ref{lem:mainHN} are vacuous here.  The second condition follows from the fact that $\Phi(\cO_p)$ is a $\theta$-stable $G$-equivariant sheaf.  Hence, $\alpha_{\Phi_0}$ is a tilt away from $\beta_{\sfZ_G}$.  We turn to Lemma \ref{lem:reverseHN}.  Now, if $\cF$ is a sheaf on $E$ with $\rk_E(\cF) =0$ then $\dim(\cF) \leq 1$.  Hence $\Phi^2(\cF) = 0$ and $\Phi^1(\cF) \in \bT_{\sfZ_G}$.  Hence (1) of Lemma \ref{lem:reverseHN} is satisfied.  Condition (2) is satisfied by construction of $\beta_{\sfZ_G}$.  Finally, we must check condition (3).  So assume that $\cF$ is a sheaf on $\cM$ supported on $E$ such that $\theta(\cF) = 0$ and $\Phi(\cM) \in \cA_{\beta_{\sfZ_G}}$.  Then $\Phi^0(\cF) \in \bF_{\sfZ_G}$, $\Phi^1(\cF) \in \bT_{\sfZ_G}$, and $\Phi^2(\cF) = 0$.  Therefore $\theta(\cF) = \theta( \Phi^0(\cF) ) - \theta( \Phi^1(\cF) )$.  Since $\theta$ is non-positive on $\bF_{\sfZ_G}$ and positive on $\bT_{\sfZ_G}$ we see that in fact $\theta( \Phi^0(\cF) ) = \theta( \Phi^1(\cF) ) = 0$ and therefore $\Phi^1(\cF) = 0$.  Moreover, $\Phi^0(\cF)$ is $\theta$-semistable and therefore there are finitely many $p \in E$ such that $\Hom(\Phi^0(\cF),\Phi^0(\cO_p)) = \Hom(\cF,\cO_p) \neq 0$.  This implies that $\cF$ is a sheaf of finite length.

Now we observe that we can adapt the argument of \ref{lem:mainHN} to show that for any sheaf $\cF$ on $\cM$, $\Phi^0(\cF) \in \bF_\theta$.  It remains to show that for any sheaf $\cF$, $\Phi^2(\cF) \in \bT_\theta$.  To this end we note that the fibers of $\cM_\theta \to \C^3/G$ are at most one-dimensional away from the origin.  Therefore, for any sheaf $\cF$, $\Phi^2(\cF)$ is supported at the origin.  We recall that $\Phi^2(\cF) = \bR^2 p_{\C^3 *}( \cE_\theta \tensor p_{\cM_\theta}^*\cF )$.  The support of $\cE_\theta$ is projective over $\C^3$.  Let $Z \subset \C^3 \times \cM_\theta$ be the support of $\cE_\theta$.  Then $Z$ is finite over $\cM_\theta$ and projective over $\C^3$.  Moreover, we claim that $Z \times_{\C^3} \{0\} = Z \times_{\cM_\theta} E_0$ as closed sets, where $E_0$ is the (reduced) exceptional divisor over $0$ in $\C^3/G$.  Indeed, if $\cE$ is a $\theta$-stable $G$-constellation on $T \times \C^3$ with support $Z_T$ then $Z_T \subset T \times \{0\}$ if and only if the induced map $T \to \cM_\theta$ factors through $E_0 \to \cM_\theta$. 

Now, we return to $\Phi^2(\cF)$.  Let $\cI_0$ be the ideal of zero in $\C^3$.  Since $\Phi^2(\cF)$ has finite length and is supported at $0$ it is $\cI_0$-adically complete and the Theorem on formal functions (see e.g. \cite{Ha}) implies that
\[ \Phi^2(\cF) = \invlim{\, \rmH^2(Z, \cE_\theta \tensor p_{\cM_\theta}^*\cF \tensor \cO_Z/\cI_0^n ) }.\] 
On the other hand if $\wh{Z}$ is the completion of $Z$ along $Z \times_{\C^3} \{0\} = Z \times_{\cM_\theta} E_0$ and $\cI_{E_0}$ is the ideal sheaf of $E_0$ then
\[ \invlim{\, \rmH^2(Z, \cE_\theta \tensor p_{\cM_\theta}^*\cF \tensor \cO_Z/\cI_0^n ) } = \rmH^2(\wh{Z}, \cE_\theta \tensor p_{\cM_\theta}^*\cF \tensor \cO_{\wh{Z}} ) =  \invlim{\, \rmH^2(Z, \cE_\theta \tensor p_{\cM_\theta}^*(\cF \tensor \cO_{\cM_\theta}/\cI_{E_0}^n) ) }. \] 
For each $n$, $\cE_\theta \tensor p_{\cM_\theta}^*(\cF \tensor \cO_{\cM_\theta}/\cI_{E_0}^n)$ is supported on $E_0$ and hence $\rmH^2(Z, \cE_\theta \tensor p_{\cM_\theta}^*(\cF \tensor \cO_{\cM_\theta}/\cI_{E_0}^n) ) \in \bT_\theta$.  We conclude that $\Phi^2(\cF) \in \bT_\theta$. 
\end{proof}

\begin{rmk}  
Suppose that $G \subset \SL_2(\C)$.  We can think of $G$ as a finite subgroup of $\SL_3(\C)$ for example by using a splitting $\C^3 = \C^2 \oplus \C$ and having $G$ act trivially on the second factor.  Then $\C^3/G \cong \C^2/G \times \C$ is a transverse singularity.  In this situation the t-structures $\alpha_\Phi$ and $\beta_\pi$ of Theorem \ref{thm:main.McKay} agree. 
\end{rmk}

\begin{rmk}
In the three dimensional derived McKay correspondence, the t-structure $\alpha_\Phi$ on $\bD(\C^3)^G$ can be very nontrivial.  One guess for how to describe this t-structure explicitly would be to adapt the construction of perverse (coherent) sheaves and attempt to define $\alpha_\Phi$ by restricting the possible cohomologies.  This is especially appealing in light of the results in \cite{CCL}.  However, it turns out that this is not generally the right description.  Consider $G = \mu_3$, the center of $\SL_3(\C)$.  Then $X=G\Hilb(\C^3)$ is naturally isomorphic to the blow-up of $\C^3/G$ at the singular point.  It can then be identified with the total space of $\omega_{\P^2}$.  We will show that there do not exist full subcategories $\cat{A}_0,\cat{A}_1,\cat{A}_2$ of $\coh(\C^3)^G$ such that $\cF^\bt \in \bD(\C^3)^G$ has the form $\cF^\bt = \Phi(\cG)$ if and only if $\sfH^i(\cF^\bt) \in \cat{A}_i$ for $i=0,1,2$ and $\sfH^i(\cF^\bt) = 0$ for $i \neq 0,1,2$.  If this were the case then the full subcategory $\Phi(\coh(X))$ would be closed under taking cohomology sheaves in the sense that if $\cE^\bt \in \Phi(\coh(X))$ then $\cH^0(\cE^\bt),\cH^1(\cE^\bt)[-1]$ and $\cH^2(\cE^\bt)[-2]$ all belong to $\Phi(\coh(X))$ as well.

Viewing $X$ as the total space of $\omega_{\P^2}$, let $E \cong \P^2$ be the zero section.  Then if $\cE$ is the universal $G$-cluster on $\C^3 \times X$, we identify $pr_{X *} \cE$ as $p^*(\cO \oplus \cO(1) \oplus \cO(2))$ where $p:X \to \P^2$ is the line bundle structure map.  Now let $p \in E$ be a point and $\cI_p \subset \cO_E$ the ideal sheaf on $E$ of $p$.  From the exact sequence
\[ 0 \to \cI_p(-3) \to \cO_E(-3) \to \cO_p \to 0,\] 
we see that $\Phi^0(\cI_p(-3)) = 0$ while $\Phi^1(\cI_p(-3)) = \Phi^0(\cO_p)$.  If $\Phi(\coh(X))$ were closed under taking cohomology then it would have to contain both $\Phi^0(\cO_p)$ and $\Phi^0(\cO_p)[-1]$.  This is impossible since $\Phi(\coh(X)) \cap \Phi(\coh(X))[-1] = \{0\}$. 
\end{rmk}

\section{Tilting equivalences}
Let $X$ be a variety.  A \emph{tilting bundle} $\cE$ on $X$ is a vector bundle such that (i) $\Ext^i(\cE,\cE) = 0$ for $i > 0$, and (ii) the zero sheaf is the only sheaf $\cF$ such that $\Ext^i(\cE,\cF) = 0$ for all $i$.  A tilting bundle gives rise to a tilting equivalence \cite{BvdB}.  This is a pair of inverse equivalences
\[ 
\xymatrix{
\bD(X) \ar@<2pt>[r]^{\bR\Hom(\cE,?)} & \ar@<2pt>[l]^{? \tensor^{\bL} \cE} \bD(\sfA) 
}
\]
where $\sfA = \End(\cE)$.  

The notion of a tilting bundle is related to the notion of a full exceptional sequence.  An object $\cF \in \bD(X)$ is \emph{exceptional} if $\End(\cF)=\bk$ and $\Ext^i(\cF,\cF)=0$ for $i \neq 0$.  A \emph{full exceptional sequence} is a sequence $\cF_1,\dotsc,\cF_n$ such that each $\cF_i$ is exceptional and $\Ext^i(\cF_j,\cF_k)=0$ whenever $j > k$, and the smallest thick subcategory of $\bD(X)$ containing $\cF_1,\dotsc,\cF_n$ is $\bD(X)$.  Finally, a full exceptional sequence $\cF_1,\dotsc,\cF_n$ is \emph{strong} if in addition $\Ext^i(\cF_j,\cF_k) = 0$ for all $i \neq 0$.  If $\cF_1,\dotsc,\cF_n$ is a full, strong exceptional sequence consisting of vector bundles then $\cE = \cF_1 \oplus \dotsm \oplus \cF_n$ is a tilting bundle.  See \cite{Bo} for a discussion of tilting in this special case.

We will investigate the structure of these equivalences in the case where $X$ is a surface.  It is known that every rational surface admits a tilting bundle \cite{HP}.  However, the converse is a well-known open question:

\begin{openquestion}
Is every a smooth projective surface which admits a tilting bundle rational?
\end{openquestion}

Let $X$ be a smooth projective surface with a tilting bundle $\cE$ and set $\sfA = \End(\cE)$.  Write $\Phi:\bD(X) \to \bD(\sfA)$ for $\bR\Hom(\cE,?)$.  Denote the length function on $\rmod\sfA$ by $\ell$.  Consider a weak central charge $\sfZ = \theta + i \ell$ on $\rmod\sfA$.  We will assume that $\cE$ does not have repeated indecomposable summands so that every simple $\sfA$-module is one dimensional.  Then the isomorphism classes of one-dimensional simple modules are in bijection with indecomposable idempotents of $\sfA$ and form a basis for $K_0(\sfA)$, the Grothendieck group of \emph{finite dimensional modules}.  Let $e_1,\dotsc,e_m$ be the indecomposable idempotents of $\sfA$ and $S_1,\dotsc,S_m$ the corresponding simple modules.  Then given a finite dimensional $\sfA$-module $M$ the class of $M$ in $K_0(\sfA)$ is $\sum_{i=1}^m{ \dim_\bk(M e_i) [S_i] }$.  We call the tuple $\dim(M) = ( \dim_\bk(Me_i) )$ the dimension vector of $M$.  So we can regard $\sfZ$ as a complex valued function on the set of integral dimension vectors.   We say that $\sfZ$ is \emph{compatible} if for each point $p \in X$, $\Phi(\cO_p)$ is a $\theta$-stable representation of $\sfA$ with $\theta(\Phi(\cO_p)) = 0$.  Let $\alpha$ and $\beta$ be the standard t-structures on $\bD(X)$ and $\bD(\sfA)$, respectively.

\begin{thm} \label{thm:main.tilting}
If $\rmod\sfA$ admits a weak central charge $\sfZ = \theta + i \ell$ compatible with $X$, then $X$ is rational.  Moreover, if $\sfZ_X = -\theta + i \rk$ then $\Phi$ identifies $\alpha_{\sfZ_X}[1]$ with $\beta_{\sfZ}$.  
\end{thm}
\begin{proof}
We prove the second claim first.  Note that the tilting equivalence $\Phi$ is left exact and satisfies GV and SV.  So we will apply Lemma \ref{lem:mainHN} with $Z = X$. Now since $\ell(x) = 0$ if and only if $x = 0$ conditions (1) and (3) of the Lemma are vacuous.  Condition (2) holds by assumption in this case.  Hence $\Phi^0(\cF) \in \bF_\sfZ$ and $\Phi^2(\cF) \in \bT_\sfZ$ for any coherent sheaf $\cF$ on $X$.  Next, we apply Lemma \ref{lem:reverseHN} to $\beta_\sfZ$.  Since a torsion sheaf has dimension at most one, if $\cF$ is torsion then $\Phi(\cF) \in \cA_{\beta_\sfZ}$.  By construction $-\theta$ is non-negative on $\cA_{\beta_\sfZ}$.  So we have to check that if $\Phi(\cF) \in \cA_{\beta_\sfZ}$ and $\theta(\cF)=0$ then $\cF$ is torsion.  In fact, if $\Phi(\cF) \in \cA_{\beta_\sfZ}$ then $\theta(\cF) = \theta(\Phi^0(\cF)) - \theta(\Phi^1(\cF))$.  Now, since $\Phi^0(\cF) \in \bF_\sfZ$, $\theta(\Phi^0(\cF)) \leq 0$ and since $\Phi^1(\cF) \in \bT_\sfZ$, if $\Phi^1(\cF) \neq 0$ then $\theta(\Phi^1(\cF)) > 0$.  So we conclude that $\Phi^1(\cF) = 0$ and $\theta(\Phi^0(\cF))=0$.  This implies that $\Phi^0(\cF)$ is semistable and therefore there are only finitely many $p \in X$ such that $\Hom(\Phi^0(\cF),\Phi^0(\cO_p))=\Hom(\cF,\cO_p)$ is nonzero.  Hence $\cF$ has finite support and in particular it is torsion.

We now establish the rationality of $X$.  Let $K_0(X)_{\leq 0}$ be the subgroup of $K_0(X)$ generated by the classes of sheaves of finite length.  Consider the morphism $\Div(X) \to K_0(X)/K_0(X)_{\leq 0}$ defined by $D \mapsto [ \cE \tensor \cO(D) ]$, for $D$ effective.  We observe that if $D$ is effective, then $[\cE \tensor \cO(D) ] = [\cE] + [\cE \tensor \cO_D]$ and in $K_0(X)/K_0(X)_{\leq 0}$, $[ \cE \tensor \cO_D] = \rk(\cE) [ \cO_D ]$.  Now the map $D \mapsto [ \cO_D ]$ defines an injective homomorphism $\Cl(X) \to K_0(X)/K_0(X)_{\leq 0}$.  Since $\theta(\cO_p) = 0$ for any point $p \in X$, $\theta$ defines a function on $K_0(X)/K_0(X)_{\leq 0}$.  We see that $\alpha( D ) = \theta( \cE \tensor \cO(D) ) - \theta(\cE )$ defines a homomorphism $\alpha:\Cl(X) \to \Z$.

Now we observe that for any effective divisor $D$, $\Phi(\cE \tensor \cO_D)$ belongs to $\cA_{\beta_{\sfZ}}$.  Hence $\theta( \cE \tensor \cO_D ) \leq 0$.  On the other hand $\Phi(\cE) = \Phi^0(\cE) = \sfA[0]$ and $\Phi(\cE \tensor \omega_X) = \Phi^2(\cE \tensor \omega_X)[-2] = \sfA^\vee[-2]$.  Therefore $\Phi(\cE) \in \bF_{\sfZ}$ and $\Phi(\cE \tensor \omega_X) \in \bT_{\sfZ}[2]$.  Hence $\theta(\cE) < 0$ while $\theta(\cE \tensor \omega_X) \geq 0$. We conclude that $\alpha(\omega_X) > 0$.

Since $K_0(X) \cong K_0(\sfA)$ it is free of finite rank and $K_0(X)_{\leq 0} = \Z\cdot [\cO_p]$ for any point $p \in X$.  Thus, $\mathrm{NS}(X) \cong K_0(X)_{\leq 1}/K_0(X)_{\leq 0} \tensor_\Z \R$, where $K_0(X)_{\leq 1}$ is the subgroup generated by sheaves of dimension at most 1.  We extend $\alpha$ to a linear map $\mathrm{NS}(X) \to \R$ and note that since $\alpha \leq 0$ on effective divisors but $\alpha(\omega_X) > 0$ the canonical divisor of $X$ is not pseudoeffective.  In particular the Kodaira dimension of $X$ is $-\infty$.  Next we note that since $\cO_X$ is a summand of $\sEnd(\cE)$ and $\rmH^i(\sEnd(\cE))=0$ for $i > 0$ the irregularity of $X$ is zero.  By the Kodaira-Enriques classification of surfaces, $X$ is rational.
\end{proof}

It is not known if a compatible weak central charge always exists.  Bergman and Proudfoot studied the problem in \cite{BP} with the aim of giving a GIT construction of any variety that admits a tilting bundle.  They use the term `great' where we use the term compatible.  In general the question of whether a set of modules can all be made stable is very subtle.  For example, this can be impossible if we consider partial tilting bundles, that is, vector bundles satisfying (i) but not (ii) in the definition.  
\begin{example}[Lutz Hille]
Let $B_q \P^2$ be the blow-up of $\P^2$ at $q$ and let $X$ be the blow-up of $B_q \P^2$ at a point on the exceptional divisor $E_1$ of $B_q \P^2 \to \P^2$.  Let $f:X \to B_q\P^2$ be the blowing-up map, $E_2$ the exceptional divisor and $E'_1$ the strict transform of $E_1$.  Then the cohomology class in $\rmH^1( \cO(E'_1) )$ defines an exact sequence
\[ 0 \to f^*\cO(E_1) \to \cE \to \cO(E_2) \to 0.\] 
The vector bundle $\cE$ satisfies $\Ext^i(\cE,\cE) = 0$ for $i = 1,2$ and $\End(\cE) \cong \bk[x]/(x^2)$.  The algebra $\bk[x]/(x^2)$ has one simple module and therefore the only module that is ever $\theta$-stable for some $\theta$ is the simple module.
\end{example}

If $\cE$ is a tilting bundle then this type of pathology cannot occur.  Indeed, for $p,q \in X$, 
\[\Hom_\sfA( \Hom(\cE,\cO_p),\Hom(\cE,\cO_q)) = \begin{cases} \bk & p = q, \\ 0 & p \neq q. \end{cases}\] 
Hence there is no proper quotient module of $\Hom(\cE,\cO_p)$ that ever appears as a submodule of $\Hom(\cE,\cO_q)$ for any $q$.  For a discussion of stable quiver representations with many interesting examples including examples of modules with trivial endomorphism ring that cannot be made stable, see \cite{Re}.  

We now turn to the tilting bundles constructed by Hille and Perling, to show that these bundles do fit into the framework of Theorem \ref{thm:main.tilting}.  They are defined inductively starting with a full strong exceptional sequence of line bundles on a minimal rational surface.  We will describe some of the features of their construction and refer the reader to \cite{HP} for details.  Suppose that $X$ is a rational surface and 
\[ X = X_n \stackrel{f_n}{\to} X_{n-1} \stackrel{f_{n-1}}{\to} \dotsm \stackrel{f_2}{\to} X_1 \stackrel{f_1}{\to} X_0\] 
is a sequence of blow-ups constructing $X$ from a minimal rational surface $X_0$.  Hille and Perling use this data to construct tilting bundles $\cE_i$ on $X_i$.  For each $i$ let $E_i$ be the exceptional divisor of $f_i$.  Then 
\[ \Ext^2(\cO(E_i),f_i^*\cE_{i-1}) = 0, \quad \text{and} \quad \Ext^\bt(f_i^*\cE_{i-1},\cO(E_i)) = 0\] 
but $\Ext^1(\cO(E_i),f_i^*\cE_{i-1}) \neq 0$.  More precisely Hille and Perling construct $\cE_{i-1}$ so that it has a unique indecomposable summand $\cE_{i-1}'$ such that $\Ext^1(\cO(E_i),f_i^*\cE_{i-1}') \neq 0$ and moreover this Ext group is one dimensional.  So there is also a unique extension
\[ 0 \to f^*\cE_{i-1}' \to \cF_i \to \cO(E_i) \to 0.\] 
Then they put $\cE_i = f^*\cE_{i-1} \oplus \cF_i$ and show that it is a tilting bundle on $X_i$.  

\begin{thm} \label{thm:HPexample}
Let $X$ be a rational surface and let $\cE$ be one of Hille and Perling's tilting bundles.  Then $\sfA = \End(\cE)$ admits a compatible weak central charge.
\end{thm}
\begin{proof}
Our approach is based on an idea of Bergman and Proudfoot (see \cite{BP}).  Let $\cE = \cE_1 \oplus \dotsm \oplus \cE_m$ be the decomposition of $\cE$ into indecomposable summands and let $e_1,\dotsc,e_m \in \sfA$ be the corresponding projectors.  Suppose that $M$ is an $\sfA$-module such that $M e_1$ is one dimensional and generates $M$.  Then $M$ is stable with respect to $\theta$ defined by
\[ \theta(S_1) = \dim_\bk(M) - 1, \quad \text{and} \quad \theta(S_i) = -1 \quad (i = 2,\dotsc,m).\] 
Indeed, since $M e_1$ generates $M$ it will generate every quotient.  So if $M \onto M''$ is a nonzero quotient then 
\[\theta(M'') = \dim_\bk(M) + 1 - \sum_{i=2}^m{ \dim_\bk(M'') } > \dim_\bk(M) + 1 - \sum_{i=2}^m{ \dim_\bk(M) } = 0.\] 

Suppose $X$ is minimal.  Then $\cE$ is the direct sum of the line bundles in a standard full strong exceptional collection.  It is straightforward to check that there is line bundle $\cL$ of $\cE$ such that for each point $p \in X$, $\Hom(\cE,\cO_p)$ is generated by $\Hom(\cL,\cO_p)$.  

Now, we proceed by induction on the Picard rank.  For a non-minimal rational surface $X$ constructed by blowing up a minimial surface we look at the last blow-up.  Suppose $f:X \to X'$ is the blow-up of a single point and that $\cE = f^*\cE_{X'} \oplus \cF$ as above, where there is an indecomposable summand $\cE_{X'}'$ of $\cE_{X'}$ and an exact sequence
\[ 0 \to f^*\cE'_{X'} \to \cF \to \cO(E) \to 0\] 
where $E$ is the exceptional divisor of $f$.  

Notice that since $\cO(E)|_E \cong \cO_E(-1)$ we see that $\Hom(f^*\cE_{X'},\cO) \to \Hom(f^*\cE_{X'},\cO(E))$ is an isomorphism.  Hence any map $f^*\cE_{X'} \to \cF$ has to factor through $f^*\cE'_{X'} \to \cF$ along $E$.  For $p \in E$ we have the exact sequence
\[ 0 \to \Hom(\cO(E),\cO_p) \to \Hom(\cF,\cO_p) \to \Hom(f^*\cE'_{X'},\cO_p) \to 0.\] 
Thus the one-dimensional subspace $\Hom(\cO(E),\cO_p) \subset \Hom(\cE,\cO_p)$ is in fact an $\sfA$-submodule. 

Let $g:X' \to X_0$ be the map to a minimal rational surface used to construct $\cE_{X'}$.  By induction, there is a line bundle summand of $\cE_{X_0}$ such that $\Hom(g^*\cL,\cO_q)$ generates $\Hom(\cE_{X'},\cO_q)$ for all $q \in X'$.  Then for any point $p \in X$ the submodule of $\Hom(\cE,\cO_p)$ generated by $\Hom(f^*g^*\cL,\cO_p)$ contains the submodule $\Hom(f^*\cE_{X'},\cO_p)$.  Let $M$ be the cokernel of $\Hom(f^*g^*\cL,\cO_p)\tensor_\bk \sfA \to \Hom(\cE,\cO_p)$.  If $e$ is the idempotent corresponding to the indecomposable summand $\cF$ of $\cE$ then we see that $M = M e$.  Hence $M$ admits a one dimensional quotient, which must be isomorphic to the one-dimensional submodule $\Hom(\cO(E),\cO_q) \subset \Hom(\cE,\cO_q)$ for any point $q \in E$.  This implies if $M \neq 0$ then there is a $q \in E$, $q \neq p$ and a nonzero map $\sfA$-module map $\Phi_\cE(\cO_p) = \Hom(\cE,\cO_p) \to \Phi_\cE(\cO_q) = \Hom(\cE,\cO_q)$.  However, since $\cE$ is a tilting bundle, $\Hom( \Phi_\cE(\cO_p),\Phi_\cE(\cO_q))=0$ if $p \neq q$.  We conclude that $\Hom(\cE,\cO_p)$ is always generated by the one dimensional space $\Hom(f^*g^*\cL,\cO_p)$ and therefore that there exists a weak central charge $\sfZ$ on $\rmod\sfA$ compatible with the tilting equivalence.
\end{proof}

We conclude with a result of independent interest.  If $X$ is a surface with a tilting bundle $\cE$ that decomposes as a direct sum of line bundles, then we can prove directly that it is rational.

\begin{thm} \label{thm:linebundles}
Let $X$ be a smooth projective surface with a tilting bundle $\cE$ that is a direct sum of line bundles.  Then $X$ must be rational.
\end{thm}
\begin{proof}
Suppose that $\cE = \oplus_{i=1}^r \cO(D_i)$.  Since $\cE$ is a tilting bundle, $\rmH^i( \cO(D_j - D_k) ) = 0$ for all $j,k$ and $i = 1,2$.  Suppose that $i,j$ are such that $h^0(\cO(D_i - D_j) ) \neq 0$ and let $D \in | D_i - D_j |$.  Write $D = \sum_{i=1}^m{ C_i }$ where $C_i$ are distinct irreducible curves.  Since $h^i(\cO_X) = h^i(\cO(D)) = 0$ for $i > 0$ we see that $h^1( \cO_D ) = 0$.  Now, consider the exact sequence
\[ 0 \to \cO_D \to \bigoplus_{i=1}^m \cO_{C_i} \to \cG \to 0.\] 
Since $\dim(\cG) = 0$, we find that $h^1( \cO_{C_i} ) = 0$.  So $g_a(C_i) = 0$.  Thus if $\ol{C}_i \subset C_i$ is the reduced induced subscheme, $\ol{C}_i$ also has arithmetic genus zero and is thus rational.  Now, if $h^0(\cO(D_i-D_j)) > 1$ then $D$ must have a moving component and $X$ is covered by rational curves.  By the classification of surfaces, $X$ is a blow up of a ruled surface over a curve $C$.  Then Orlov's theorem on blowups \cite{Or92} implies that that the map $K_0(C) \to K_0(X)$ induced by derived pullback is injective.  However since $K_0(X)$ is torsion free and $K_0(C)$ has torsion unless $C = \P^1$, we find that $X$ is a blowup of a rational ruled surface and hence $X$ is rational.

So it remains to show that for some pair $i,j,$ $h^0(\cO(D_j - D_i)) > 1$.  By a Reimann-Roch computation due to Hille and Perling \cite[Lemma 3.3]{HP11}, if $h^0(\cO(D_j - D_i)), h^0(\cO(D_k - D_j)) > 0$ then 
\[ h^0(\cO(D_k - D_i) ) = h^0( \cO(D_k - D_j) ) + h^0( \cO(D_j - D_i ) ).\] 
Suppose, for a contradiction, that for all $i,j$, $h^0(\cO(D_j - D_i)) \leq 1$.  If $h^0(\cO(D_j - D_i)) \neq 0$ then for all $k$, $h^0(\cO(D_k - D_j))=0$.  Let $Q$ be the quiver with vertices $\{1,\dotsc,r\}$ and with a single edge $i \to j$ whenever $h^0(\cO(D_j - D_i)) = 1$.  Then $Q$ has no paths of length 2.  Now, we note that $\End(\cE)$ is isomorphic to a quotient of the path algebra $\bk Q$ where the kernel is contained in the span of the paths of length at least two.  Hence $\End(\cE) \cong \bk Q$.  However, the global dimension of $\bk Q$ is 1, while the minimum possible global dimension of $\End(\cE)$ is 2 (see \cite{BF}).  So we see that for at least one pair $i,j$, $h^0(\cO(D_j - D_i)) > 1$ and hence $X$ is rational.
\end{proof}

\begin{rmk}
There is another result in this direction.  Bondal and Polishchuk \cite{BoPo} have shown that if $X$ is a smooth $n$-dimensional projective variety that admits a full exceptional collection of length $n+1$ (the minimum possible length) then $X$ is a Fano variety.  
\end{rmk}

\begin{rmk}
The McKay correspondence may be viewed as a tilting equivalence.  Let $G \subset \SL_n(\C)$ be a finite subgroup.  Then $G$ acts on $S = \C[x_1,\dotsc,x_n]$ and we can form the twisted group ring $S \rtimes G$.  The category of $G$-equivariant sheaves on $\C^n$ is naturally equivalent to the category of left modules over $S \rtimes G$.  The equivariant sheaf corresponding to the free module $S \rtimes G$ is $\cO \tensor \C G$.  One can check for $n = 2,3$, the inverse equivalence $\bD(\C^n)^G \cong \bD(\cM_\theta)$ carries $\cO \tensor \C G$ to $\cF = p_{\cM_\theta *}\cE$, where $\cE$ is the universal $\theta$-stable $G$-constellation.  Since $\cE$ is flat over $\cM_\theta$ we see that $\cF$ is a tilting bundle on $\cM_\theta$.  The dual of a tilting bundle is also a tilting bundle.  So we can interpret $\Phi$ as $\bR\Hom( \cF^\vee,? )$.
\end{rmk}


\bibliographystyle{alpha}
\bibliography{tstructures.bib}
\end{document}